\documentclass[11pt]{article}
\usepackage[T1]{fontenc}
\usepackage[utf8]{inputenc}
\usepackage{fullpage}
\usepackage{amsmath, amsthm, amsfonts, amssymb, amstext, mathrsfs, enumerate}
\usepackage{graphicx, ragged2e, lscape, framed, xcolor}
\usepackage{subfiles}
\usepackage{comment}

\theoremstyle{plain}
\newtheorem{theorem}{Theorem}[section]
\newtheorem{lemma}[theorem]{Lemma}
\newtheorem{proposition}[theorem]{Proposition}
\newtheorem{conjecture}[theorem]{Conjecture}
\newtheorem{corollary}[theorem]{Corollary}

\newtheorem{claim}{Claim}[section]
\newtheorem*{remark}{Remark}
\newtheorem{example}{Example}[section]

\numberwithin{equation}{section}
\allowdisplaybreaks

\newcommand{\affl}[3]{\noindent #1, Email: {\tt #2}\\ \textsc{#3}\\[1.5pt]}

\usepackage[pagebackref]{hyperref}
\hypersetup{
	colorlinks=true,
    urlcolor=purple,
	linkcolor=purple,
    citecolor=purple,
}

\DeclareMathOperator{\tr}{tr}
\DeclareMathOperator{\diam}{diam}

\DeclareMathOperator{\dist}{dist}

\DeclareMathOperator{\col}{col}
\newcommand{\av}[1]{{#1}_{\text{av}}}

\def\e{\mbox{\boldmath $e$}}

\def\uvec{\mbox{\boldmath $u$}}
\def\vvec{\mbox{\boldmath $v$}}

\def\x{\mbox{\boldmath $x$}}
\def\z{\mbox{\boldmath $z$}}
\def\1{\mbox{\boldmath $1$}}

\title{\textbf{Clique and chromatic numbers of graph powers: \\spectral bounds and anticodes}}

\author{Aida Abiad,\, Gabriel Coutinho,\, Hitesh Kumar,\, Emanuel Juliano,\, Michael Tait}
\date{}

\begin{document}
\maketitle
\begin{abstract}
Using spectral and polynomial methods, we show several sharp spectral bounds for the clique and chromatic numbers of graph powers that only use spectral information from the original graph. In particular, we strengthen earlier spectral-radius estimates to obtain an asymptotically sharp colouring bound for general graph powers, and prove a sharp bound for powers of line graphs with a complete characterisation of equality. We also develop polynomial methods yielding upper and lower 
bounds for clique numbers of graph powers. Regarding the connection between the clique and coloring parameters of graph powers, we propose a spectral analogue of Wegner’s conjecture, determine the extremal values for graph squares, and establish the conjecture for all powers when the spectral-radius bound is sufficiently small. Finally, we connect our spectral polynomial method with coding theory: we prove that for distance-regular graphs, optimisation over the adjacency algebra recovers Delsarte's linear programming bound for anticodes.
\end{abstract}



\section{Introduction}
Let $G=(V,E)$ be a finite simple graph and let $k$ be a positive integer. The \emph{$k$-th power} $G^k$ has vertex set $V$, with two distinct vertices adjacent if and only if their distance in $G$ is at most $k$. A \emph{distance-$k$ colouring} of $G$ assigns different colours to vertices at distance at most $k$, whereas a \emph{distance-$k$ clique} is a set of vertices whose pairwise distances are at most $k$. Their associated parameters satisfy
\[
    \chi_k(G)=\chi(G^k),\qquad
    \omega_k(G)=\omega(G^k),\qquad
    \omega_k(G)\leq\chi_k(G).
\]
Thus, graph powers turn distance-constrained colouring and packing questions into questions about familiar graph invariants.
These parameters have both combinatorial and metric interpretations. In a network of processors, $G^k$ records which pairs of processors can communicate within $k$ rounds, so a clique in $G^k$ is a set of processors that are pairwise reachable within that time; see~\cite{Linial_1992} for the role of locality in distributed graph algorithms. A distance-$k$ clique is also an anticode of diameter at most $k$ in the graph metric, connecting the problem to coding theory and code--anticode bounds; see, for example,~\cite{Shi_Xia_Krotov_2023}. More broadly, distance constraints underlie a variety of packing and optimisation problems, while every distance-$k$ clique is an immediate obstruction to using fewer colours in a distance-$k$ colouring.

Powers of line graphs provide the corresponding edge-distance framework. If $L(G)$ denotes the line graph of $G$, then a colouring of $L(G)^k$ assigns different colours to edges at distance at most $k$ in $L(G)$. In particular, $\chi_2(L(G))$ is the \emph{strong chromatic index} of $G$, and $\omega_2(L(G))$ is its \emph{strong clique index}. The chromatic and clique numbers of graph and line graph powers have received substantial attention; see the survey~\cite{Cranston_2023}, the graph-power colouring result~\cite{Pierron_2024}, and recent work on distance-edge colourings, strong cliques, strong edge-colourings, and the degree--diameter problem~\cite{Abiad_Reijnders_2026,Kumar_Mohar_Pragada_2026,Davey_Hurley_Verclos_Kang_Volec_2026,vanBatenburg_Korsky_2026}. Two central conjectures are Wegner's conjecture comparing the extremal clique and chromatic numbers of graph powers under a maximum-degree constraint (see Conjecture~\ref{conj:Wegner}), and the Erd\H{o}s--Ne\v{s}et\v{r}il conjecture on strong edge-colouring; see~\cite[Sections~2 and~4]{Cranston_2023}. 

Spectral, algorithmic, algebraic, and probabilistic methods have also been developed for other aspects of graph powering. These include bounds for the independence number~\cite{Atkinson_Frieze_2004,Abiad_Cioaba_Tait_2016,Abiad_Coutinho_Fiol_2019,Abiad_Zhou_2026} and the chromatic number~\cite{Alon_Mohar_2002,Kierstead_Yang_Yi_2020,Abiad_Coutinho_Fiol_Nogueira_Zeijlemaker_2022, Abiad_Zhou_2026}, as well as optimisation and complexity theory questions~\cite{Agnarsson_Halldorsson_2003,Molloy_Salavatipour_2005,Abiad_Hoesel_Michaux_2026,Chalermsook_Laekhanukit_Nanongkai_2014}. Independence and colouring problems in powers of random graphs have likewise been studied~\cite{Atkinson_Frieze_2004,Garapaty_Lokshtanov_Maji_Pothen_2023,Frieze_Kang_Raut_Sweering_Verbeek_2026}.

From the spectral graph theoretic perspective, bounds on the independence and chromatic number of graph powers have been applied to estimate the Shannon capacity~\cite{Fiol_2020, Abiad_Dalfo_Fiol_2026}, the largest size of codes under certain metrics \cite{APR2025,AbiadNeriReijnders2025}, edge-isoperimetric numbers~\cite{Abiad_vandeBerg_Juliano_Reijnders_Simoens_vanVeluw_Wittebol_2026}, and quantum analogues of the distance chromatic and independence numbers~\cite{Abiad_Jany_2025,Abiad_Meeus_2025,Wocjan_Elphick_Abiad_2022}. 

A main difficulty is that the spectrum of $G$ does not generally determine the adjacency matrix of $G^k$ through a polynomial identity. Matrix powers $A(G)^j$ count walks, whereas adjacency in $G^k$ records whether a sufficiently short path exists. Consequently, applying a classical spectral bound directly to $G^k$ need not produce a bound in terms of the spectral data of $G$; see~\cite[Sections~2 and~3]{Abiad_Coutinho_Fiol_Nogueira_Zeijlemaker_2022} or \cite{Das_Guo_2013, Xing_Zhou_2017}. Nor can spectral bounds for distance independence simply be transferred to distance cliques by complementation: in general, $(\overline{G})^k\neq\overline{G^k}$, and hence $\alpha_k(G)$ need not equal $\omega_k(\overline{G})$. These structural obstacles accompany computational ones: even at $k=1$, clique number and chromatic number are NP-hard to compute, and graph-power colouring has substantial hardness-of-approximation barriers~\cite{Chalermsook_Laekhanukit_Nanongkai_2014}.

In this paper, we derive multiple bounds for $\chi_k(G)$ and $\omega_k(G)$ from spectral information about the original graph, without requiring the spectrum of its power. We use two complementary features of adjacency polynomials. For upper bounds, suitable walk-counting polynomials dominate the adjacency matrix of the power entrywise. For lower bounds, a polynomial of bounded degree has columns supported on small metric balls. The first feature yields spectral-radius and expander-mixing estimates; the second relates clique size to quadratic forms and spectral moments.

We begin in Section~\ref{sec:spectral_radius} by improving estimates for $\lambda_1(G^k)$ and $\lambda_1(L(G)^k)$. Combined with Wilf's inequality, these give the colouring bounds of Section~\ref{sec:wilfspectrarelated}. For general graph powers, Theorem~\ref{thm:wilf_graph_power_improved} strengthens the earlier spectral-radius estimate by an explicit term involving the minimum degree; its leading term is asymptotically sharp in view of the constructions in~\cite{vanBatenburg_Korsky_2026}. For line graph powers, Theorem~\ref{thm:wilf_line_graph power} gives a bound in terms of the largest signless Laplacian eigenvalue of $G$, together with a complete characterisation of equality. These colouring estimates also extend to the corresponding list-colouring and painting parameters.

For regular graphs, we then use the nonprincipal eigenvalues to bound clique size through a one-sided polynomial expander-mixing argument (Theorem~\ref{thm:EML_approach}). Non-backtracking walk polynomials give explicit choices in this bound, and we show that the bound is sharp for even cycles and incidence graphs of projective planes, and also for the squares of line graphs of subdivisions of complete graphs, where a different dominating polynomial is used. These examples illustrate the strength of the same polynomial expander-mixing bound for both graph and line graph powers.

In Section~\ref{sec:lower_bounds}, we obtain lower bounds for $\omega_k(G)$, and therefore also for $\chi_k(G)$, from local balls and spectral moments. Theorem~\ref{thm:general_lower_bound} applies to arbitrary graphs and real polynomials of suitable degree. It requires neither regularity nor a polynomial identity for $A(G^k)$. Choosing the all-ones vector or a scaled Perron vector gives explicit consequences involving the average degree, the spectral radius, and the concentration of the Perron vector.

In Section~\ref{sec:Wegner} we propose a spectral analogue of Wegner's extremal conjecture. The constraint now concerns the spectral radius rather than the maximum degree. We determine the extremal clique and chromatic numbers for $k=1,2$. We also prove their equality for all powers throughout the range $1\leq r\leq\sqrt{2+\sqrt{5}}$. Our results show that the extremal behaviour under a spectral-radius constraint can differ from that under a maximum-degree constraint.

We conclude in Section~\ref{sec:coding} with a coding-theoretic interpretation of our spectral results. For distance-regular graphs, optimizing the polynomial expander-mixing bound over the full adjacency algebra recovers exactly Delsarte's linear-programming bound (LP) for anticodes. This gives an equally strong spectral formulation, while the spectral bound itself also applies without the need to use an association scheme. We illustrate this connection with an application of the spectral method to nested linear codes and a sharp BCH coset-graphs.

\section{Notation and background}
Throughout this paper, we use standard graph terminology and notation. Let $G=(V(G), E(G))$ be a simple undirected graph of order $n = |V(G)|$ and size $m = |E(G)|$. Let $A(G)$ denote the \emph{adjacency matrix} of $G$ with eigenvalues enumerated as $\lambda_{1}(G)\geq \cdots \geq \lambda_{n}(G)$. If $G$ is clear from the context, we will omit $G$ and simply write $\lambda_i$ to mean $\lambda_i(G)$. 

Throughout, $\omega(G)$, $\alpha(G)$ and $\chi(G)$ will denote the \emph{clique number}, \emph{independence number} and the \emph{chromatic number} of $G$. For a vertex $u\in V(G)$, we denote its \emph{degree} in $G$ by $\deg_G(u)$ (or simply $\deg(u)$ if $G$ is clear from context). We denote the \emph{maximum degree}, \emph{minimum degree}, \emph{diameter} and \emph{girth} of $G$ by $\Delta(G)$, $\delta(G)$, $\diam(G)$ and $g(G)$, respectively. The distance between $u,v\in V(G)$ in $G$ is denoted by $d_G(u,v)$. We write $u\sim v$ if $u$ and $v$ are adjacent and $u\nsim v$ otherwise. For a subset $S\subseteq V(G)$, we denote by $G[S]$ the subgraph of $G$ \emph{induced} by the vertices in $S$. A \emph{clique} $C$ in $G$ is a subset of $V(G)$ such that $G[C]$ is a complete graph. We denote the \emph{line graph} of $G$ by $L(G)$. 

For a positive integer $k$, the \emph{k-th graph power} of $G$, denoted by $G^k$, is the graph which has the same vertex set as $G$, with two vertices adjacent in $G^k$ if and only if they are at distance at most $k$ in $G$. For a vertex $u\in V(G)$, we define the \emph{$k$-ball centred at $u$} to be the set \[ B_k(u): = \{v\in V(G): \dist(u,v)\le k\}.\] Equivalently, $B_k(u)$ is the closed neighbourhood of $u$ in $G^k$. A \emph{distance-k clique} of $G$ is a clique in $G^k$. We define \emph{distance-$k$ clique number} $\omega_k(G)$ to be the clique number of $G^k$, i.e., $\omega_k(G):=\omega(G^k)$. The \emph{distance-$k$ independence number} $\alpha_k(G)$ and the \emph{distance-$k$-chromatic number} $\chi_k(G)$ are defined analogously. 

We denote by $\mathbb{R}[x]$ (resp. $\mathbb{R}_k[x]$) the set of polynomials (resp. with degree at most $k$) having real coefficients. Let $A$ and $B$ be square matrices of the same order. We write $A\le B$ whenever $A$ is entry-wise less than or equal to $B$. We denote by $I$, $J$ and $\mathbf{1}$ the identity matrix, the all-ones matrix and the all-ones vector of appropriate order. 

\section{Spectral radius of graph powers}\label{sec:spectral_radius}

The spectral radius of a graph almost universally appears in any spectral estimation of graph parameters. Hence, it is only natural to investigate the spectral radius of graph powers first.

To apply Wilf's inequality to $G^k$, we first need to control $\lambda_1(G^k)$ using information about $G$. Adjacency powers count walks, so shortest paths give entrywise comparisons between $A(G^k)$ and suitable expressions in $A(G)$. The useful comparisons also remove contributions from immediate returns and repeated encounters with nearby vertices. We begin with the inequalities of Miao and Fan \cite{Miao_Fan_2014}, and then obtain a quantitative improvement for higher powers. 

For a graph $G$, let $A(G)$, $D(G)$ and $\mathcal{L}(G) := D(G) - A(G)$ denote its \emph{adjacency}, \emph{degree} and the \emph{Laplacian} matrix. The following result compares the adjacency matrix of powers of $G$ to the adjacency matrix of $G$.  

\begin{lemma}[\cite{Miao_Fan_2014}]\label{lemma:adjacency_graph_powers} Let $G$ be a connected graph.
\begin{enumerate}
    \item Then $A(G^2)\le A(G)^2 - \mathcal{L}(G)$ with equality if and only if girth $g(G)\ge 5$. 
    \item If $k\ge 3$, then 
    \begin{align*}
    A(G^k) & \le A(G^{k-1})A(G) - A(G)(D(G)-I) - \mathcal{L}(G),\\
    A(G^k) & \le A(G)A(G^{k-1}) - (D(G)-I)A(G) - \mathcal{L}(G).
    \end{align*}
\end{enumerate}
\end{lemma}

As a consequence, Miao and Fan \cite{Miao_Fan_2014} established the following.

\begin{lemma}[\cite{Miao_Fan_2014}]\label{lemma:spectral_radius_powers}
    Let $G$ be a connected graph of order $n\ge 3$. 
    \begin{enumerate}[$(i)$]
        \item Then $\lambda_1(G^2)\le \lambda_1^2$ with equality if and only if $G^2$ is regular and girth $g(G)\ge 5$. 
        \item If $k\ge 3$, then $\lambda_1(G^k) < \lambda_1^k$. 
    \end{enumerate}
\end{lemma}

We improve the estimate for $\lambda_1(G^k)$ when $k\ge 3$.

\begin{theorem}\label{thm:spectral_radius_higher_power}
Let $k\ge 3$ and $G$ be a connected graph of order $n\ge 4$ with minimum degree $\delta(G)$. Then 
\[\lambda_1(G^k) \le \lambda_1^k - \delta(G)\frac{\lambda_1^{k-2}-1}{\lambda_1 -1}.\]
\end{theorem} 

\begin{proof} 
    We first prove the following recursive bound:
    \begin{equation}\label{eq:recursive_bound}
        \lambda_1(G^k)\le \lambda_1 \lambda_1(G^{k-1}) - \delta(G).
    \end{equation}
    So consider the unit Perron vector $\x$ corresponding to $\lambda_1(G^k)$. Using Lemma \ref{lemma:adjacency_graph_powers}, we see that 
    \begin{align*}
        \lambda_1(G^k) 
        & = \x^\top A(G^k) \x \\
        & \le \x^\top A(G)A(G^{k-1})\x - \x^\top (D(G)-I)A(G)\x - \x^\top \mathcal{L}(G)\x\\
        & \le \lambda_1\lambda_1(G^{k-1}) - f_G(\x), 
    \end{align*}
    where 
    \begin{align*}
      f_G(\x):& = \sum_{uv\in E(G)} \left(\deg(u)+\deg(v)-2\right)x_ux_v + \sum_{uv\in E(G)} (x_u-x_v)^2 \\
      & = \sum_{uv\in E(G)} \left(\deg(u)+\deg(v)-4\right)x_ux_v + \sum_{u\in V(G)} \deg(u)x_u^2.
    \end{align*}
    Thus, to prove \eqref{eq:recursive_bound}, we only need to show that $f_G(\x)\ge \delta(G)$. We consider the following cases:

    \textbf{Case 1:} $\delta(G)\ge 2$. ~ Then $\deg(u)+\deg(v)-4\ge 2\delta(G)-4\ge 0$. Thus, 
    \[ f_G(\x) \ge \sum_{u\in V(G)} \deg(u)x_u^2 \ge \delta(G)\sum_{u\in V(G)}x_u^2  = \delta(G). \]

   \textbf{Case 2:} $\delta(G) = 1$. ~ Since $G$ is connected and $n\ge 4$, we see that $\deg(u) + \deg(v)\ge 3$ for any edge $uv\in E(G)$. So consider the set 
   \[F = \{ab\in E(G): \deg(a)=1 \text{ and }\deg(b) =2\}.\]
   Observe that $F$ is a matching in $G$; indeed if $ab, bc\in F$, then $\deg(a) = \deg(c) = 1$ and $\deg(b)=2$, forcing $G\cong P_3$ which contradicts the assumption that $n\ge 4$.  
   
   Let $ab\in F$ with vertex $a$ being a leaf in $G$. Then 
   \[ (\lambda_1(G^k)+1)x_a = x_a + \sum_{ua \in E(G^k)} x_u \le x_b+\sum_{ub\in E(G^k)}x_u = (\lambda_1(G^k)+1)x_b,\]
   which implies $x_a\le x_b$. Thus, we see that 
   \begin{align*}
       f_G(\x) & \ge -\sum_{ab \in F}  x_ax_b + \sum_{u\in V(G)}\deg(u)x_u^2\\
       & \ge -\sum_{\deg(b)\ge 2}  x_b^2 + \sum_{u\in V(G)}\deg(u)x_u^2 \quad (\text{since $F$ is a matching})\\
       & \ge \sum_{u\in V(G)}x_u^2 = 1.
   \end{align*}
   We conclude that \eqref{eq:recursive_bound} holds. Now, iterating \eqref{eq:recursive_bound}, we get 
   \begin{align*}
       \lambda_1(G^k) & \le \lambda_1^{k-2}\lambda_1(G^2) - \delta(G)\sum_{i=0}^{k-3} \lambda_1^i \\
       & = \lambda_1^{k-2}\lambda_1(G^2) - \delta(G)\frac{\lambda_1^{k-2}-1}{\lambda_1-1}\\
       &\le \lambda_1^k - \delta(G)\frac{\lambda_1^{k-2}-1}{\lambda_1-1},
   \end{align*}
 where the last inequality follows by Lemma \ref{lemma:spectral_radius_powers}. This completes the proof.
\end{proof}

Next, we investigate the spectral radius of powers of line graphs; for these the structure inherited from the original edges permits a sharper comparison. Given a graph $G$, there is a natural connection between the signless Laplacian matrix $Q(G)=D(G)+A(G)$ and the adjacency matrix $A(L(G))$ of the line graph $L(G)$. The natural spectral parameter is the largest eigenvalue $q_1(G)$ of $Q(G)$. If $R(G)$ is the vertex--edge incidence matrix of $G$, then
\begin{equation}
    R(G)R(G)^\top = Q(G)\qquad \text{and}\qquad R(G)^\top R(G) = A(L(G)) + 2I.
\end{equation}
It is clear that $\lambda_1(L(G)) = q_1(G)-2$. By Lemma \ref{lemma:adjacency_graph_powers} it is therefore immediate that 
\begin{equation}
    \lambda_1(L(G)^k) \le \lambda_1(L(G))^k = (q_1(G)-2)^k.  
\end{equation}
Using the structure of line graphs, this upper bound can be substantially improved. Indeed, this generic bound treats $L(G)$ as an arbitrary graph. Instead, a shortest path in $L(G)$ can be represented by a non-backtracking walk in $G$. Counting these walks through the arc incidence structure leads to the following bound. For a $d$-regular graph, its right-hand side becomes $2\sum_{i=1}^k(d-1)^i$, the corresponding count in a tree-like edge neighbourhood; the equality condition identifies exactly when that count is attained.

\begin{theorem}\label{thm:spectral_radius_line_graph_power} Given a non-empty connected graph $G$ and $k\ge 1$, 
 \[ \lambda_1(L(G)^k)\le 2\sum_{i=1}^k \left(\frac{q_1(G)-2}{2}\right)^i.\]  
 Equality holds if and only if $k=1$ or $k\ge 2$ and $G$ is regular with girth $g(G)\ge 2k+1$. 
\end{theorem}

\begin{proof} For the graph $G$, define its set of arcs to be the set 
\[\overrightarrow{E}(G):=\{(u,v), (v,u): uv\in E(G)\}.\]
Define the non-backtracking matrix $B$ indexed by arcs as follows:
\[ B_{(u,v), (x,y)}:= 
\begin{cases}
    1 & v = x \text{ and }y\ne u;\\
    0 & \text{otherwise}.
\end{cases}\]
Let $S$ denote the arc-edge incidence matrix defined as follows: for $(u,v)\in \overrightarrow{E}(G)$ and $e\in E(G)$, 
\[ S_{(u,v), e} :=
\begin{cases}
    1 & e = \{u,v\};\\
    0 & \text{ otherwise}.
\end{cases}\]
Clearly, $S^\top S = 2I$. For $i\ge 1$, note that the entry $(S^\top B^iS)_{e,f}$ $(e,f\in E(G))$ counts non-backtracking walks in $G$ consisting of $i+1$ edges with first edge $e$ and last edge $f$. Every shortest path of length $i$ in $L(G)$ corresponds to such a non-backtracking walk. Hence, 
\begin{equation}\label{eq:adjacency_line_comparison}
   A(L(G)^k)\le \sum_{i=1}^k S^\top B^iS. 
\end{equation}
Moreover, 
\begin{equation}\label{eq:B_and_L_G}
    (B + B^\top)S = SA(L(G)). 
\end{equation}
It is therefore clear that $\lambda_1(B+B^\top) = \lambda_1(A(L(G))) = q_1(G)-2$. Let $\x$ denote the unit Perron vector of $A(L(G)^k)$. Then $\|S\x\|^2 = 2$ since $S^\top S = 2I$. 

For a complex $k\times k$ matrix $M$ (not necessarily symmetric), its numerical radius is given by 
\[w(M):=\max_{\|\z\|=1}|\z^* M\z|,\quad \z\in \mathbb{C}^k.\]
By Berger's power inequality (cf. \cite{Berger_Stampfli_1967}), it is known that $w(M^i)\le w(M)^i$. Also, if $M$ is real and non-negative, then $w(M) = \lambda_1\left(\frac{M+M^\top}{2}\right)$. Using these facts, we see that 
\begin{align*}
    \lambda_1(L(G)^k) & = \x^\top A(L(G)^k)\x\\
    & \le \sum_{i=1}^k (S\x)^\top B^i (S\x)\\
    & \le \|S\x\|^2 \sum_{i=1}^k  w(B^i)\\
    & \le 2 \sum_{i=1}^k  w(B)^i\\
    & = 2\sum_{i=1}^k \lambda_1\left(\frac{B+B^\top}{2}\right)^i \\
    & = 2\sum_{i=1}^k \left(\frac{q_1(G)-2}{2}\right)^i. 
\end{align*}
This proves the inequality. 

We now analyze the equality case. Clearly, equality holds when $k=1$; so assume $k\ge 2$. Then we must have
\[ (S\x)^\top B (S\x) = 2 \lambda_1\left(\frac{B+B^\top}{2}\right),\]
which implies that $\x$ is a Perron eigenvector for $A(L(G))$ by \eqref{eq:B_and_L_G}. We must also have
\[ (S\x)^\top B^2 (S\x) = 2 \lambda_1\left(\frac{B+B^\top}{2}\right)^2. \]
But 
\[(S\x)^\top B^2 (S\x) = \left\|\frac{B+B^\top}{2}S\x\right\|^2 - \left\|\frac{B-B^\top}{2}S\x \right\|^2.\]
On comparison, we get 
\[ \frac{B-B^\top}{2}S\x  = 0.\]
Hence, for an arc $(u,v)$, 
\[\left(\sum_{ e\ni v} x_e \right) - x_{uv} = (BS\x)_{(u,v)} = (B^\top S\x)_{(u,v)} = \left(\sum_{ e\ni u} x_e \right) - x_{uv}.\]
By the connectedness of the graph, the quantity $\sum_{ e\ni v} x_e$ is a constant, say $s$, for all vertices $v\in V(G)$. By the eigenvalue equation, 
\[\lambda_1(L(G))x_{uv} + 2x_{uv} = \sum_{ e\ni v}x_e + \sum_{ e\ni u}x_e =2s.\]
Thus, $x_{e}$ is a constant, say $r$, for all edges $e\in E(G)$. It follows that 
\[s = \sum_{e\ni u} x_e = \deg_G(u)r,\]
implying that $\deg_G(u)$ is a constant, i.e., $G$ is regular. 

Finally, equality in \eqref{eq:adjacency_line_comparison} is equivalent to $g(G)\ge 2k+1$. Combining these observations gives the characterization for equality. 
\end{proof}

\section{Upper bounds} \label{sec:wilfspectrarelated}

\subsection{Wilf-type upper bound for graph and line graph powers}

We first recall the degree bounds against which the spectral estimates can be compared. The greedy colouring bound, with the equality characterization supplied by Brooks' theorem, gives, for every connected graph $G$,
\begin{equation}\label{eq:Brooks}
  \chi(G)\le 1 + \Delta(G), 
\end{equation}
with equality if and only if $G$ is a complete graph or an odd cycle.

For given positive integers $k$ and $\Delta$, define 
\[ D_{k, \Delta}: = \max\ \{\Delta(G^k): \Delta(G)\le \Delta\},\]
i.e., $D_{k, \Delta}$ is the maximum degree of all $k$-powers of graphs with maximum degree at most $\Delta$. It is easy to see that 
\[ D_{k, \Delta} = \sum_{i=1}^k \Delta(\Delta - 1)^{i-1} =
\begin{cases}
  2k & \text{ if }\Delta =2;\\
  \Delta \frac{(\Delta-1)^k - 1}{\Delta - 2} & \text{ if }\Delta \ge 3.
\end{cases} \]

This maximum is attained by a rooted tree of depth $k$ in which the root has $\Delta$ children and every other non-leaf vertex has $\Delta-1$ children. Thus, \eqref{eq:Brooks} gives
\[\chi_k(G)\le 1 + D_{k, \Delta}\]
for any graph $G$ with maximum degree $\Delta$. But this can be improved. Indeed, the following theorem improves this degree-only estimate by exploiting the structure of a graph power. They provide a useful comparison with the spectral bounds below, especially in the regular case.

\begin{theorem} The following hold for a connected graph $G$ with $\Delta(G) = \Delta \ge 3$.  
\begin{enumerate}[$(i)$]
    \item (\cite{Cranston_Rabern_2016}) If $G^2 \neq K_{\Delta^2 +1}$, then $\chi_2(G)\le \Delta^2 -1 = D_{2, \Delta}-1$.
    \item (\cite{Bonamy_Bousquet_2014}) If $k\ge 3$, then $\chi_k(G)\le D_{k, \Delta} - 1$.
\end{enumerate}
\end{theorem}

The following result of Wilf \cite{Wilf_1967} strengthens \eqref{eq:Brooks} since $\lambda_1(G)\le \Delta(G)$.
 
\begin{theorem}[Wilf's inequality \cite{Wilf_1967}]\label{thm:Wilf}
For any connected graph $G$, we have 
    \[ \chi(G) \le 1 + \lambda_1(G).\]
Equality holds if and only if $G$ is a complete graph or an odd cycle.  
\end{theorem}

An analogous result for graph powers was established by Miao and Fan \cite{Miao_Fan_2014}.

\begin{theorem}[\cite{Miao_Fan_2014}]\label{thm:wilf_type_bound_graph_powers}
        Let $G$ be a connected graph of order $n\ge 3$. 
    \begin{enumerate}[$(i)$]
        \item Then $\chi_2(G)\le 1 + \lambda_1^2$ with equality if and only if $G^2 = K_{\lambda_1^2 + 1}$, i.e., $G$ is one of the following graphs: a star, the pentagon, the Petersen graph, the Hoffman–Singleton graph, and putative $57$-regular graphs on $3250$ vertices. 
        \item If $k\ge 3$, then $\chi_k(G) < 1 + \lambda_1^k$. 
    \end{enumerate}
\end{theorem}

Using Theorems \ref{thm:spectral_radius_higher_power} and \ref{thm:Wilf}, we immediately get the following improvement when $k\ge 3$. 

\begin{theorem}\label{thm:wilf_graph_power_improved}
     Let $G$ be a connected graph of order $n\ge 4$ and let $k\ge 3$. Then
     \[\chi_k(G) \le 1 + \lambda_1^k - \delta(G)\frac{\lambda_1^{k-2}-1}{\lambda_1 -1}.\]
\end{theorem}

For any fixed $k$, Cames van Batenburg and Korsky \cite{vanBatenburg_Korsky_2026} recently constructed regular graphs $H_{k,q}$ with $\diam(H_{k,q}) = k$ and order $(1+o(1))\Delta(H_{k,q})^k$. Clearly, then $\lambda_1(H_{k,q}) = \Delta(H_{k,q})$ and $H_{k,q}^k$ is a complete graph, and so 
\[ \frac{\chi_k(H_{k,q})}{\lambda_1(H_{k,q})^k}\to 1 \text{ as }q\to \infty.\]
In particular, the bound in Theorem \ref{thm:wilf_graph_power_improved} is asymptotically sharp.

Whenever $\lambda_1 \le \Delta(G)-1$, Theorem \ref{thm:wilf_graph_power_improved} gives a much better upper bound than $D_{k, \Delta} - 1$ for $\chi_k(G)$. On the other hand, when $G$ is regular, $\lambda_1 = \Delta(G)$, and so in this case $D_{k, \Delta}-1$ is a better bound. 

The \emph{coloring number} of a graph $G$, denoted by $\chi_{\col}(G)$, is defined to be 
\[ \chi_{\col}(G):= 1 + \max_{H\subseteq G}\delta(H).\]
Let $\chi_\ell(G)$ and $\chi_p(G)$ denote the \emph{list chromatic number} and the \emph{paint number} of $G$, respectively. It is well-known that (see, for instance, \cite{Cranston_2023})
\[ \chi(G)\le \chi_\ell(G)\le \chi_p(G)\le \chi_{\col}(G)\le 1 + \Delta(G).\]
Since $\delta(H)\le \lambda_1(H)\le \lambda_1(G)$ for any subgraph $H$ of $G$, we see that $1 + \lambda_1$ is indeed an upper bound for $\chi_{\col}(G)$ (this is precisely the proof of Wilf's inequality). Thus, all of the above spectral bounds for $\chi_k$ also extend to distance-$k$ analogues of list chromatic number and paint number of graphs. 

For line graphs, we obtain the following Wilf-type inequality as a consequence of Theorem \ref{thm:spectral_radius_line_graph_power}. 

\begin{theorem}\label{thm:wilf_line_graph power}
Given a connected graph $G$ of order $n\ge 3$ and $k\ge 1$, 
 \[ \chi_k(L(G))\le 1 + 2\sum_{i=1}^k \left(\frac{q_1(G)-2}{2}\right)^i.\] 
Equality holds if and only if one of the following occurs:
\begin{enumerate}[$(i)$]
    \item $k=1$ and $G$ is a star or an odd cycle. 
    \item $k\ge 2$ and $G\cong C_{2k+1}$. 
\end{enumerate}
\end{theorem}

\begin{proof} The inequality is immediate by Wilf's inequality and Theorem \ref{thm:spectral_radius_line_graph_power}. We now analyze the equality case. Suppose equality holds for some graph $G$. We consider the following cases:

\textbf{Case 1:} $k=1$. ~ By Theorem \ref{thm:Wilf}, we must have that $L(G)$ is either a complete graph or an odd cycle, which implies $G$ is either a star or an odd cycle. It is easy to check that equality holds in both of these cases. 

\textbf{Case 2:} $k\ge 2$. ~ By Theorem \ref{thm:spectral_radius_line_graph_power}, $G$ must be regular and $g(G)\ge 2k+1$. By Theorem \ref{thm:Wilf}, we must further have that $L(G)^k$ is complete as $L(G)^k$ cannot be an odd cycle under these hypotheses. We thus have $q_1(G) = 2\Delta$ where $\Delta = \Delta(G)$, $\chi_k(L(G)) = m$ where $m = |E(G)|$, and 
\[ m = 1 + 2 \sum_{i=1}^k (\Delta(G)-1)^i. \]
Modulo $\Delta$, we have
\begin{equation}
    m \equiv 
    \begin{cases}
    1 \mod \Delta & \text{ if } k \text{ is even};\\
    -1 \mod \Delta & \text{ if } k \text{ is odd}.
    \end{cases}
\end{equation}
Since $2m = n\Delta$, we conclude that $\Delta$ divides $2$. If $\Delta = 1$ then $G\cong K_2$. If $\Delta = 2$, then $G\cong C_n$. Since $g(G)\ge 2k+1$ and $\diam(G)\le k$, we conclude that $n = 2k+1$, i.e., $G\cong C_{2k+1}$. Conversely, it is clear that equality holds when $G\cong C_{2k+1}$. 
\end{proof}

Note that at $k=2$ this is an upper bound for the strong chromatic index in terms of $q_1(G)$. The equality characterization is more restrictive than for the spectral-radius bound itself: for higher powers, equality in both steps forces an odd cycle of the appropriate length.

\subsection{Upper bound using Expander Mixing Lemma}

For regular graphs, the nonprincipal eigenvalues can give information about cliques that is lost in a bound using only the spectral radius. Suppose that $p(A(G))$ dominates $A(G^k)$ entrywise. On a clique of $G^k$, the associated quadratic form must then be large. An expander-mixing argument bounds the same form by separating the constant eigenvector from its orthogonal complement. Comparing the two estimates gives the following theorem. Note that the hypothesis is polynomial domination rather than a polynomial identity for $A(G^k)$.

\begin{theorem}[Polynomial EML bound]\label{thm:EML_approach}
Let $G$ be a connected $d$-regular graph of order $n\ge 4$ and $p(x)\in \mathbb{R}_t[x]$ $(t\ge k)$ be such that $A(G^k)\le p(A(G))$. Let $\beta:=\max_{2\le i\le n} p(\lambda_i)$. If $n - p(\lambda_1) +  \beta > 0$, then
\[
\frac{\omega_k(G)}{n} \le \frac{1 + \beta}{n - p(\lambda_1) + \beta}.
\]
\end{theorem}

\begin{proof}
Let $\{\mathbf{x}_i \}$ be an orthonormal set of eigenvectors of $A(G)$ where $\mathbf{x}_i$ corresponds to $\lambda_i$. Consider a maximum clique $C$ in $G^k$. Let $\mathbf{1}_C \in \mathbb{R}^{V(G)}$ denote the indicator vector for $C$. Write 
    \[
    \mathbf{1}_C = \sum_{i=1}^n \alpha_i \mathbf{x}_i \quad \text{where}\quad  \alpha_i = \langle \mathbf{1}_C, \mathbf{x}_i \rangle.
    \]
Since $G$ is regular, we have that $\alpha_1 = \frac{|C|}{\sqrt{n}}$ and so
    \[
    \sum_{i=2}^n\alpha_i^2 = |C| - \frac{1}{n}|C|^2.
    \]    
Now, observe that 
    \[
    \mathbf{1}_C^T\ p(A(G))\ \mathbf{1}_C = \sum_{i=1}^n \alpha_i^2 p(\lambda_i) \le \frac{1}{n}|C|^2p(\lambda_1) + \left(|C| - \frac{1}{n}|C|^2 \right)\beta.
    \]
Also
    \[
    \mathbf{1}_C^T \ p(A(G))\ \mathbf{1}_C \geq \mathbf{1}_C^T\ A(G^k)\ \mathbf{1}_C \ge |C|(|C|-1).
    \] 
Combining the two inequalities gives
\[ |C|^2\left(1 - \frac{p(\lambda_1) -  \beta}{n}\right) - |C|\left(1 + \beta \right)\le 0.\]
Solving for $|C|$, we get 
 \[ \frac{|C|}{n} \le \frac{1 + \beta}{n - p(\lambda_1) +  \beta}.\]
 This completes the proof.
\end{proof}

One naturally wonders what the optimal choice of a polynomial is in Theorem \ref{thm:EML_approach}. One immediate choice for $t\ge k$ is $p(x) = x^t + x^{t-1}$. Indeed, if $u,v$ are distinct vertices with $\ell=\dist(u,v)\le t$, then the shortest $u-v$ path in $G$ can be extended to a walk of length either $t$ or $t-1$ depending on the parity of $\ell$, simply by retracing the path one edge back and forth. Thus, \[A(G^k)_{uv}\le 1 \le  \left(A(G)^t + A(G)^{t-1}\right)_{uv},\]
which shows that $A(G^k)\le p(A(G))$. 

A better choice comes from non-backtracking polynomials. For a $d$-regular graph $G$, define 
\[ F_0(x) := 1, \qquad F_1(x):=x, \qquad F_2(x) := x^2-d, \]
and for $j\ge 3$ define 
\[ F_j(x) := xF_{j-1}(x)-(d-1)F_{j-2}(x).\]
It is well-known that $F_j(A(G))_{uv}$ is exactly the number of non-backtracking walks of length $j$ from $u$ to $v$ in $G$, see  e.g. \cite{Alon_Benjamini_Lubetzky_Sodin_2007}. For $t\ge k$, let
\begin{equation}\label{eq:non_backtracking}
   p_t(x):=\sum_{j=1}^t F_j(x). 
\end{equation}
It is clear that $A(G^k)\le p_t(A(G))$ and so Theorem \ref{thm:EML_approach} can be applied with polynomial $p_t(x)$. 

We now mention some families that are tight for Theorem \ref{thm:EML_approach}. 

\textbf{Cycles:} Let $G = C_{2k+2}$. Since $g(G)=2k+2>2k$, it is clear that $p_k(A(G)) = A(G^k)$ where $p_k$ is as defined in \eqref{eq:non_backtracking}. Moreover, $G^k$ is the graph obtained from $K_{2k+2}$ by deleting a perfect matching. Hence, $\omega_k(G) = k+1$. The spectrum of $G^k$ is $\{2k^{(1)}, 0^{(k+1)}, (-2)^{(k)}.\}$ So for the polynomial $p_k$, we have $p_k(2)=2k$, $\beta = 0$ and the upper bound from Theorem \ref{thm:EML_approach} is 
\[ \omega_k(G)\le \frac{n}{n-2k} = \frac{2k+2}{2} = k+1.\]

\textbf{Incidence graphs of projective planes:} Let $G$ be the incidence graph of a projective plane of order $q$. Then it is known that $G$ is $(q+1)$-regular bipartite graph of order $2(q^2+q+1)$ with girth $6$, diameter $3$ and spectrum $\{\pm  (q+1)^{(1)}, \pm \sqrt{q}^{(q^2+q)}\}$. For $k=2$, \[p_2(x) = x^2 + x - (q+1)\] 
by \eqref{eq:non_backtracking}. Since $g(G)>4$, we have $A(G^2) = p_2(A(G))$. It is easy to check that $p_2(q+1) = (q+1)^2$ and $\beta = p_2(-q-1)= q^2-1$. Thus, by Theorem \ref{thm:EML_approach}, we have 
\[ \omega_2(G)\le 2(q^2+q+1)\frac{1+(q^2-1)}{2(q^2+q+1) - (q+1)^2 + (q^2-1)}= q^2+q+1.\]
Conversely, since $G$ is bipartite and has diameter $3$, the bipartition classes form a clique in $G^2$ and so $\omega_2(G)\ge q^2+q+1$. Therefore, the inequality in Theorem \ref{thm:EML_approach} is sharp for $G$. 

Next, we give a tight example which uses a polynomial different from $p_t$ defined in \eqref{eq:non_backtracking}.

\textbf{Line graphs of subdivisions of complete graphs:} Let $S(K_r)$ denote the graph obtained from $K_r$ by subdividing all of its edges exactly once. Define $H_r:= L(S(K_r))$. Then $H_r$ is an $(r-1)$-regular graph of order $r(r-1)$. We claim the following.

\begin{claim}
For $r\ge 4$ and $k=2$, $\omega_2(H_r) = r$ and Theorem \ref{thm:EML_approach} is sharp for $H_r$ with polynomial 
\[ f_r(x) := \frac{x(x+2)(x-r+2)(2x-r+1)}{r+1}.\]
\end{claim}

\begin{proof}
Identify the vertices of $H_r$ with ordered pairs $(i,j)$, where $i\ne j$ and $1\le i,j\le r$. Then adjacency in $H_r$ is given by:
\[ (i,j)\sim (j,i)\quad \text{and}\quad (i,j)\sim (i,a), \quad a\notin \{i,j\}.\]
We will first argue that $A(H_r^2)\le f_r(A(H_r))$. Let $\e_{ij}$ be the indicator vector for vertex $(i,j)$, and define the row and column vectors to be 
\[ \uvec_i:= \sum_{a\ne i}\e_{ia}, \qquad \vvec_i := \sum_{a\ne i}\e_{ai}.\]
The adjacency relation in $H_r$ gives
\begin{equation}
    (A(H_r)+I)\e_{ij} = \uvec_i + \e_{ji},
\end{equation}
and 
\begin{equation}
    A(H_r)\uvec_i = (r-2)\uvec_i + \vvec_i, \quad A(H_r)\vvec_i = \1 - \vvec_i, \quad A(H_r)\1 = (r-1)\1. 
\end{equation}
Using the above equations, it can be checked that
\[ f_r(A(H_r))\e_{ij} = \1 - \vvec_{j}.\]
Using the adjacency relation in $H_r$, it is clear that the neighborhood of $(i,j)$ in $H_r^2$ is exactly 
\[\{(j, i)\}\cup \{(i, a), (a, i), (j, a): a\ne i,j\}.\]
Every vertex in this neighbourhood has the second coordinate different from $j$. Therefore, 
\[A(H_r^2)\e_{ij} \le \1 - \vvec_j = f_r(A(H_r))\e_{ij}.\]
This holds for every $(i,j)$ and so $A(H_r^2)\le f_r(A(H_r))$.

Now, the spectrum of $H_r$ is given by 
\[\left\{(r-1)^{(1)}, (r-2)^{(r-1)}, 0^{(r(r-3)/2)}, (-1)^{(r-1)}, (-2)^{(r-1)(r-2)/2}\right\}.\] 
The spectrum follows directly from the signless-Laplacian spectrum of subdivision graphs together with the standard incidence-matrix relation between a graph and its line graph; see \cite{Barik_Kalita_Pati_Sahoo_2018}. Therefore, $f_r(r-1) = (r-1)^2$ and $\beta = 0$. Thus, by Theorem \ref{thm:EML_approach}, we have 
\[\omega_2(H_r) \le \frac{r(r-1)}{r(r-1) - (r-1)^2} = r.\]
Conversely, $\omega_2(H_r)\ge \Delta(H_r) + 1 = (r-1) + 1 = r$. Claim follows. 
\end{proof}

\section{Lower bounds}\label{sec:lower_bounds}

We now turn our attention to deriving spectral lower bounds for $\omega_k(G)$, and hence also for $\chi_k(G)$. The basic observation is local: if $2t\le k$, then the triangle inequality makes each ball $B_t(u)$ a clique in $G^k$. To estimate the size of such a ball spectrally, we use a polynomial $p$ of degree at most $t$. The $u$-th column of $p(A(G))$ is supported inside $B_t(u)$, so Cauchy-Schwarz relates a quadratic form to the ball size. Summing over the vertices gives the trace quotient below.

This argument applies to arbitrary graphs and real polynomials with $p(A(G))\ne 0$; neither regularity nor nonnegative polynomial coefficients are required. In particular, it does not require $A(G^k)$ to be a polynomial in $A(G)$. The choice of $p$ and of the test vector allows the bound to use different parts of the spectral information.

\begin{theorem}\label{thm:general_lower_bound}
Let $G$ be a graph and $t, k \in \mathbb{N}$ be such that $2t \leq k$. Consider $p\in \mathbb{R}_t[x]$ such that $p(A(G))\ne 0$. Then, 
\[\omega_k(G)\ \ge\ \max_{\|\z\|_\infty = 1} \ \frac{\z^\top p(A(G))^2\z}{\tr(p(A(G))^2)}.
\]
\end{theorem}

\begin{proof} Let $M = p(A(G))$. Since $\deg(p(x))\le t$, we have $M_{uv}=0$ whenever $\dist(u,v)>t$. For any vertex $u\in V(G)$, let $\e_u\in \mathbb{R}^{V(G)}$ denote the vector with $1$ in the coordinate corresponding to $u$ and $0$ elsewhere. Consider a vector $\z\in \mathbb{R}^{V(G)}$ with $\|\z\|_\infty = 1$. Since $M$ is symmetric,
\[
        \e_u^\top M\z = \z^\top M\e_u = \z_{B_t(u)}^\top M\e_u,
\]
where $\z_{B_t(u)} = \z$ on $B_t(u)$ and $0$ elsewhere. By Cauchy--Schwarz inequality, 
\[
        (\z^\top M\e_u)^2 = (\z_{B_t(u)}^\top M\e_u)^2
        \le \|\z_{B_t(u)}\|_2^2\,\|M\e_u\|_2^2
        \le |B_t(u)|\,\|M\e_u\|_2^2 .
\]
Since $2t\le k$, $B_t(u)$ forms a clique in $G^k$. So using $|B_t(u)|\le \omega_k(G)$, we obtain
\[
        (\z^\top M\e_u)^2
        \le
        \omega_k(G)\,\|M\e_u\|_2^2 .
\]
Summing the above inequality over all $u\in V(G)$ gives
\[
        \z^\top M^2\z
        \le
        \omega_k(G)\sum_{u\in V(G)} \|M\e_u\|_2^2 .
\]
Finally,
\[
        \sum_{u\in V(G)} \|M\e_u\|_2^2
        =
        \tr(M^2).
\]
Therefore
\[
        \omega_k(G)
        \ge
        \frac{\z^\top M^2\z}
             {\tr(M^2)}.
\]
Since $\z$ was arbitrary, the assertion follows. 
\end{proof}

Taking $p(x)=x^t$ expresses the denominator in Theorem~\ref{thm:general_lower_bound} as the even spectral moment $\sum_i\lambda_i^{2t}$. Two test vectors give particularly simple consequences. The all-ones vector yields a bound involving the average degree, while the Perron vector gives a bound involving the spectral radius and the size of its largest coordinate. The latter retains information about how concentrated the Perron vector is. We note here that results similar in flavour can be found in the work of Fiol, see \cite{Fiol_2002, Fiol_2013}.

\begin{corollary}
Let $G$ be a non-empty connected graph and let $t, k \in \mathbb{N}$ be such that $2t \leq k$.
\begin{enumerate}[$(i)$]
    \item Let $\av d$ denote the average degree of $G$. Then 
    \[ \omega_k(G) \geq \frac{n}{\sum_{i = 1}^n \left(\frac{\lambda_i}{\av d}\right)^{2t}}.\]
    \item Let $\mathbf{x}$ denote the unit Perron-eigenvector of $G$ and $x_{\max}$ denote its maximum entry. Then
       \[\omega_k(G) \ge \frac{1}{x_{\max}^{2}\left(1 + \sum_{i=2}^n\left(\frac{\lambda_i}{\lambda_1}\right)^{2t}\right)}.\]
\end{enumerate}
\end{corollary}

\begin{proof} To prove assertion $(i)$, take $z = \mathbf{1}$ and $p(A(G)) = A(G)^t$. Clearly, 
\[ \tr(A(G)^{2t}) = \sum_{i=1}^n \lambda_i^{2t}.\]
Moreover, 
\[ \| A(G)^t \mathbf{1}\|_2^2 = \mathbf{1}^\top A(G)^{2t}\mathbf{1}.\]
Suppose $\{\mathbf{x}_1, \ldots, \mathbf{x}_n\}$ is an orthonormal set of eigenvectors of $A(G)$ where $\mathbf{x}_i$ corresponds to $\lambda_i$. Write 
\[\frac{1}{\sqrt{n}}\mathbf{1} = \sum_{i=1}^n \alpha_i \mathbf{x}_i, \quad \sum_{i=1}^n \alpha_i^2 = 1.\]
Then, using Jensen's inequality for the convex map $x \mapsto x^{2t}$, we have 
\[\frac{1}{n} \mathbf{1}^T A(G)^{2t}\mathbf{1} = \sum_{i=1}^n \alpha_i^2 \lambda_i^{2t}\ge \left(\sum_{i=1}^n \alpha_i^2 \lambda_i\right)^{2t} = \left(\frac{1}{n}\mathbf{1}^T A(G)\mathbf{1}\right)^{2t} = \av d^{2t}.\]  
Using Theorem \ref{thm:general_lower_bound}, we see that 
\[ \omega_k(G)\ge \frac{n \av d^{2t}}{\sum_{i=1}^n \lambda_i^{2t}}.\]

We now prove $(ii)$. Consider $z = \frac{\mathbf{x}}{\mathbf{x}_{\max}}$ and take $p(A(G)) = A(G)^{t}$. Then  
\[ z^Tp(A(G))^2z = \frac{\lambda_1^{2t}}{\mathbf{x}_{\max}^2}\quad \text{and}\quad  \tr(p(A(G))^2) = \sum_{i=1}^n \lambda_i^{2t}.\]
The desired inequality holds by Theorem \ref{thm:general_lower_bound}. 
\end{proof}

For regular graphs, the following lower bound for $\omega_k$ is immediate.
 
\begin{corollary}\label{cor:regular_lb}
    Let $G$ be a connected $d$-regular graph and let $t, k \in \mathbb{N}$ be such that $2t \leq k$. Then 
        \[ \omega_k(G) \geq \frac{n}{\sum_{i = 1}^n \left(\frac{\lambda_i}{d}\right)^{2t}}.\]
\end{corollary}

Suppose $G$ is a dense $d$-regular graph whose $k$-th power is complete and $n\left(\frac{\lambda}{d}\right)^{2t} = o(1)$ where $\lambda = \max_{2\le i\le n} |\lambda_i|$. Then, by Corollary \ref{cor:regular_lb}, we have 
\[\omega_k(G)\ge \frac{n}{1 + o(1)}, \]
i.e., the lower bound is asymptotically sharp. There are many such regular graph families. For instance, complete multipartite graph $K_{\underbrace{a, \ldots, a}_{s}}$ with $a$ fixed and $s\to \infty$ has 
\[n = as, \quad d = a(s-1), \quad \lambda = a, \quad \text{and}\quad n\left(\frac{\lambda}{d}\right)^{2t} = \frac{as}{(s-1)^{2t}}\to 0\quad \text{as} \quad s\to \infty.\]  

\section{Spectral variant of Wegner's Conjecture}\label{sec:Wegner}


Perhaps the most well-known, quite broad and widely open conjecture concerning graph powers is Wegner's Conjecture, posed in 1977. For integers $k\ge 1$ and $D\ge 1$, define 
\[ \omega_\Delta(k,D):=\max_{\Delta(G)\le D}\omega_k(G) \quad \text{and}\quad \chi_\Delta(k,D) :=\max_{\Delta(G)\le D}\chi_k(G).\]
Clearly, $\omega_\Delta(k,D)\le \chi_\Delta(k,D)$. Wegner proposed that the equality always holds. 

\begin{conjecture}[Wegner's conjecture \cite{Wegner_1977} cf. \cite{Cranston_2023}]\label{conj:Wegner} For $k\ge 2$ and $D\ge 3$, we have 
\[  \omega_\Delta(k,D) =\chi_\Delta(k,D).\]
\end{conjecture}

The case $D=1$ is trivial, and the case $D=2$ is easy; hence, the restriction $D\ge 3$. The case $k=1$ is also clear with the extremal graph being $K_{D+1}$. For $k\ge 2$ and $D\ge 3$, there are only finitely many pairs $(k, D)$ for which the conjecture is known to be true. We summarize the known values in Table \ref{table:Wegner}. 

\begin{table}[ht]
    \centering
    \begin{tabular}{c c c l l}
    \hline
  $k$ &  $D$ & $\chi_\Delta(k,D) =\omega_\Delta(k,D)$ & Lower bound for $\omega_\Delta(k,D)$ & Upper bound for $\chi_\Delta(k,D)$\\
  \hline
   2 & 3 & 10 & Hoffman--Singleton \cite{Hoffman_Singleton_1960} & Brooks' Theorem\\ 
   2 & 4 & 15 & Elspas \cite{Elspas_1964} & Cranston--Rabern \cite{Cranston_Rabern_2016} \\
   2 & 5 & 24 & Elspas \cite{Elspas_1964} & Cranston--Rabern \cite{Cranston_Rabern_2016} \\
   2 & 7 & 50 & Hoffman--Singleton \cite{Hoffman_Singleton_1960} & Brooks' Theorem\\
   3 & 3 & 20 & Feria-Pur\'{o}n et al. \cite{Puron_Miller_Guillermo_2011} & Bonamy--Bousquet \cite{Bonamy_Bousquet_2014} \\
   \hline
\end{tabular}
    \caption{Progress on Wegner's conjecture}
    \label{table:Wegner}
\end{table}

Motivated by the above conjecture, we naturally investigate its spectral variant. For an integer $k\ge  1$ and a real number $r\ge 1$, define 
\[ \omega_\lambda(k,r):= \max_{\lambda_1(G)\le r}\omega_k(G)\quad \text{and}\quad \chi_\lambda(k,r):= \max_{\lambda_1(G)\le r}\chi_k(G).\]
These are well-defined as $\Delta(G)\le \lambda_1(G)^2 \le r^2$ is bounded. Since $\lambda_1(G)\le \Delta(G)$, it is clear that 
\begin{equation}\label{eq:spectral_vs_degree}
   \omega_\Delta(k,D)\le \omega_\lambda(k,D)\quad \text{and}\quad \chi_\Delta(k,D)\le \chi_\lambda(k,D). 
\end{equation}

\begin{remark}
    Equality in \eqref{eq:spectral_vs_degree} may not hold in general. For instance, from Table \ref{table:Wegner} we have $\chi_\Delta(2,4) = 15$, whereas $\chi_\lambda(2,4) = 17$ from Theorem \ref{thm:spectral_Wegner_k_1_2}. 
\end{remark}

After preliminary computational investigations, we propose the following conjecture.

\begin{conjecture}\label{conj:spectraL_Wegner} For any integer $k\ge 1$ and real number $r\ge 1$, we have 
\[ \omega_\lambda(k,r) = \chi_\lambda(k,r).\]
\end{conjecture}

The restriction $r\ge 1$ is sensible because any non-empty graph $G$ has $\lambda_1(G)\ge 1$. In what follows, we provide evidence for this conjecture by proving it for various small values of $k$ and $r$. We first show that Conjecture \ref{conj:spectraL_Wegner} holds for $k=1,2$. As Conjecture \ref{conj:Wegner} remains open for $k=2$, this suggests that the two conjectures are quite different in flavour. Also, it is not difficult to see that for Conjecture \ref{conj:Wegner} there always exists an extremal $D$-regular graph $G$ with $\chi_k(G) = \chi_\Delta(k, D)$ for any $(k, D)$. The extremal graph for the Conjecture \ref{conj:spectraL_Wegner} may not always be regular.   

\begin{theorem}\label{thm:spectral_Wegner_k_1_2} For any $r\ge 1$, the following hold:
\begin{enumerate}[$(i)$]
    \item $\omega_\lambda(1,r) = \chi_\lambda(1,r) = 1 + \lfloor r\rfloor$.
    \item $\omega_\lambda(2,r) = \chi_\lambda(2,r) = 1 + \lfloor r^2\rfloor $.
\end{enumerate}
\end{theorem}

\begin{proof} Let $k=1$. Wilf's inequality gives $\chi_1(G)\le 1 + \lambda_1(G) \le 1 +r$, and so $\chi_\lambda(1,r)\le 1 + \lfloor r\rfloor$. Moreover, the graph $K_{\lfloor r\rfloor +1}$ has spectral radius $\lfloor r\rfloor$ and so $\lfloor r\rfloor +1 \le \omega_\lambda(1,r)$.
Thus, $(i)$ holds. 

Assume $k=2$. Using Wilf's inequality and Lemma \ref{lemma:spectral_radius_powers} $(i)$, we see that 
\[ \chi(G^2)\le 1 + \lambda_1(G^2)\le 1 + \lambda_1^2(G) \le 1 + r^2,\]
and so $\chi_\lambda(2,r)\le 1 + \lfloor r^2\rfloor$. Moreover, if $G =K_{1,\lfloor r^2\rfloor}$, then $\lambda_1(G)\le r$ and $G^2 = K_{\lfloor r^2\rfloor+1}$. Assertion $(ii)$ follows. 
\end{proof}

Next, we prove Conjecture \ref{conj:spectraL_Wegner} when $r\le \sqrt{2+\sqrt{5}}$. The constant $\sqrt{2+\sqrt{5}}$ is quite special as Brouwer--Neumaier \cite{Brouwer_Neumaier_1989} gave a complete classification of graphs with spectral radius at most $\sqrt{2+\sqrt{5}}$, which facilitates our proof.  

\begin{theorem} Let $k\ge 1$ and $1\le r\le \sqrt{2+\sqrt{5}}$. Then $\omega_\lambda(k,r) = \chi_\lambda(k,r)$.
\end{theorem}

\begin{proof}  It is well-known that if $G$ is a tree, then $G^k$ is a chordal graph (cf. \cite{Agnarsson_Halldorsson_2007}), and so $\omega(G^k) = \chi(G^k)$ for all $k\ge 1$. We consider the following cases:

\textbf{Case 1:} $r < 2$. ~ Then by the classification result of Smith \cite{Smith_1970} (cf. \cite[Theorem 3.1.3]{Brouwer_Haemers_2012}), any connected graph $G$ with $\lambda_1(G)<2$ is a tree. Since any power of a tree is perfect, we conclude that $\omega_\lambda(k,r) = \chi_\lambda(k,r)$.

\textbf{Case 2:} $2\le r\le \sqrt{2+\sqrt{5}}$. ~ Let $G$ be a connected graph with $\lambda_1(G)\le r$ and $\chi_k(G) = \chi_\lambda(k,r)$. A consequence of the classification result of Brouwer--Neumaier \cite{Brouwer_Neumaier_1989} (cf. \cite{Cvetkovic_Doob_Gutman_1982}) is that $G$ is either a tree or a cycle. If $G$ is a tree, then the assertion holds. 

So suppose $G$ is a cycle. Then $\chi_\lambda(k,r) =\chi_k(G)\le 1 + \Delta(G^k) \le 1 + 2k$. Note that $C_{2k+1}^k \cong K_{2k+1}$ which implies $\omega_\lambda(k,r)\ge 2k+1$. Thus, $\omega_\lambda(k,r) = \chi_\lambda(k,r)$. 
\end{proof}

Woo--Neumaier \cite{Woo_Neumaier_2007} gave a structural classification of graphs with spectral radius at most $3/\sqrt{2}$. The next natural step is to prove Conjecture \ref{conj:spectraL_Wegner} for $r\le 3/\sqrt{2}$ and $k\ge 1$. We can prove this for small values of $k$, but proving it for all $k$ remains challenging as the classification includes graphs whose powers are difficult to analyze. We leave this problem for interested readers. 

\section{A spectral formulation of Delsarte's LP method for anticodes}\label{sec:coding} 
A distance-$k$ clique is an \emph{anticode of diameter at most $k$} in the graph metric. Thus the same extremal problem describes a clique in a graph power, or a family of words with bounded pairwise distance. In this section we show that, for distance-regular graphs, optimizing Theorem~\ref{thm:EML_approach} over the full adjacency algebra loses none of the strength of Delsarte's linear-programming (LP) bound. In fact, in this section we show that the optimized polynomial EML bound from Theorem~\ref{thm:EML_approach} and the Delsarte LP bound are equivalent.

There is substantial precedent for applying Delsarte's method to maximum-distance, rather than minimum-distance, problems. Wilson's proof of the exact $t$-intersecting Erd\H{o}s--Ko--Rado bound is an application of Delsarte LP in the Johnson scheme~\cite{Wilson_1984}; these intersecting families are precisely anticodes in the Johnson graph. Tanaka develops this viewpoint for Grassmann and bilinear forms graphs~\cite{Tanaka_2006}, and gives a unified LP treatment together with an application to twisted Grassmann graphs~\cite{Tanaka_2012}. The latter also illustrates that distance-regularity need not come from vertex-transitivity. A related, but distinct, use of anticodes is the \emph{code--anticode inequality}: Martin and Zhu use Delsarte's anticode condition to exclude nontrivial perfect codes in Grassmann and bilinear forms graphs~\cite{Martin_Zhu_1995}. That application uses the inequality between the sizes of an actual code and an actual anticode; it should not be confused with optimizing an LP upper bound on anticode size.

The relation between LP and spectral ratio bounds is itself classical. Luz proves that Delsarte's bound for a graph whose edges are a union of classes of a symmetric association scheme is the minimum of suitable ratio bounds~\cite{Luz_2007}; see also the eigenvalue approach to rank-metric codes in~\cite{Abiad2024EigenvalueCodes}. Our result specializes this principle to the polynomial expander-mixing bound for graph powers and makes the conversion explicit in the adjacency algebra of the \emph{original} graph. This transfers known anticode bounds to our setting and explains why an optimized polynomial can improve substantially on a fixed walk-counting choice. We distinguish below the hypotheses for this exact equivalence from the weaker hypotheses needed merely to apply Theorem~\ref{thm:EML_approach}.

\subsection{The anticode LP and its relation to the Polynomial EML bound}
For a connected distance-regular graph $G$ on $N$ vertices and of diameter $D$, write $A_i$ for its distance matrices and $E_j$ for the primitive idempotents of its distance scheme, with $E_0=J/N$. Define the second eigenmatrix $Q$ by
\[
 E_j=\frac1N\sum_{i=0}^{D}Q_{ij}A_i.
\]
For a nonempty anticode $S$, its inner distribution is
$a_i=|S|^{-1}\1_S^{\mathsf T}A_i\1_S$. Then $a_0=1$,
$\sum_i a_i=|S|$, and
$\sum_i a_iQ_{ij}=N|S|^{-1}\1_S^{\mathsf T}E_j\1_S\ge0$.
Consequently, the standard anticode LP is~\cite{Delsarte_1973,Tanaka_2012}
\begin{equation}\label{eq:coding-primal}
 L_k(G)=\max\left\{\sum_{i=0}^{D}a_i:
 \begin{array}{l}
 a_0=1,\quad a_i\ge0\ (0\le i\le D),\\
 a_i=0\ (i>k),\quad
 \sum_{i=0}^{D}a_iQ_{ij}\ge0\ (0\le j\le D)
 \end{array}\right\}.
\end{equation}
Unlike the minimum-distance code LP, it is the \emph{large} distances that are forbidden. In particular, $L_k(G)$ is a real-valued relaxation with $\omega_k(G)\le L_k(G)$.

\begin{proposition}\label{prop:coding-delsarte}
Let $G$ be a connected distance-regular graph on $N$ vertices, of degree $d$ and diameter $D$, and let $1\le k<D$. Write $L_k(G)$ for the Delsarte linear-programming bound on the size of a diameter-$k$ anticode in its distance scheme. Then
\begin{equation}\label{eq:coding-delsarte}
 L_k(G)=\inf_p
 \frac{N(1+\beta_p)}{N-p(d)+\beta_p},
 \qquad
 \beta_p=\max_{\lambda\ne d}p(\lambda),
\end{equation}
where the maximum is over the nonprincipal adjacency eigenvalues and the infimum is over real polynomials of degree at most $D$ satisfying
$A(G^k)\le p(A(G))$ entrywise and $N-p(d)+\beta_p>0$.
\end{proposition}

\begin{proof}
Put $A=A(G)$ and let $\mathcal A$ be the real span of the distance
matrices $A_0=I,A_1=A,\ldots,A_D$. Distance-regularity gives
$\mathcal A=\{p(A):\deg p\le D\}$; every matrix in this algebra has constant diagonal and constant row sum. Let $J$ denote the all-ones matrix. The dual form of \eqref{eq:coding-primal} is
\begin{equation}\label{eq:coding-dual}
 L_k(G)=N\min\bigl\{\tau:
 F\in\mathcal A,\ F\succeq0,\ F\1=\1,\ F_{xx}=\tau,
 \ F_{xy}\le0\ \text{if }1\le\dist(x,y)\le k\bigr\}.
\end{equation}
Here $F\succeq0$ means that $F$ is positive semidefinite. To see the
bound directly, the indicator vector of an anticode $S$ satisfies
$|S|^2/N\le\1_S^T F\1_S\le\tau|S|$.

For a feasible $p$, put $\beta=\beta_p$ and
$\delta=N-p(d)+\beta$. The matrix
\[
 F=\frac{J+\beta I-p(A)}{\delta}
\]
has row sum one and nonnegative eigenvalues: its eigenvalue on $\1$ is one, and those on $\1^\perp$ are $(\beta-p(\lambda))/\delta$. Its entries at distances $1,\ldots,k$ are nonpositive, and its diagonal is at most $(1+\beta)/\delta$, since $p(A)$ has nonnegative diagonal. Thus $L_k(G)$ is at most every bound on the right of
\eqref{eq:coding-delsarte}.

Conversely, take an optimal $F$ in \eqref{eq:coding-dual}.
Since $G^k$ is not complete, its complement has positive degree $h$ and $F=(I+A(\overline{G^k})/h)/2$ is feasible with diagonal $1/2$; hence the optimal $\tau<1$.
An optimal $F$ has a zero eigenvalue on $\1^\perp$: otherwise its least
such eigenvalue $\mu$ satisfies $0<\mu<1$, since $\tau<1$, and
$(F-\mu I)/(1-\mu)$ is feasible with smaller diagonal
$(\tau-\mu)/(1-\mu)<\tau$.
For sufficiently small $\varepsilon>0$, define $p$ by
\[
 p(A)=J+(\varepsilon\tau-1)I-\varepsilon F.
\]
This matrix has zero diagonal, entries at least one at distances
$1,\ldots,k$, and nonnegative entries at all other distances. Therefore it dominates $A(G^k)$. Moreover,
\[
 \beta_p=\varepsilon\tau-1,\qquad
 p(d)=N-1+\varepsilon(\tau-1),\qquad
 \frac{N(1+\beta_p)}{N-p(d)+\beta_p}=N\tau=L_k(G),
\]
which proves the reverse desired inequality.
\end{proof}

The classical finite Delsarte LP is formulated for an association scheme; it does \emph{not} require the scheme to arise from a distance-regular graph. Prescribed allowed relations can be handled in any symmetric association scheme~\cite{Delsarte_1973,Schrijver_1979}. In Proposition~\ref{prop:coding-delsarte}, distance-regularity provides the more specific identity
\[
 \operatorname{span}\{A_0,\ldots,A_D\}=\mathbb R[A]
   =\{p(A):\deg p\le D\}.
\]
It is this identity, together with the constant diagonals and row sums in the distance algebra, that permits the \emph{full} LP to be represented by polynomials in one adjacency matrix. No $Q$-polynomial assumption, vertex-transitivity, or linearity of the anticode is used. For a general association scheme, a chosen relation graph need not generate the full Bose--Mesner algebra, so optimizing only over its adjacency polynomials need not recover the scheme's LP. Nor may we replace the degree bound $D$ by $k$ without potentially weakening the result: the theorem concerns optimization over the entire algebra, not a universal low-degree formula or a guaranteed reduction in computational complexity.

In contrast, our polynomial EML bound (Theorem~\ref{thm:EML_approach}) applies to \emph{every} regular graph for which a polynomial satisfies $A(G^k)\le p(A(G))$ and the denominator is positive. It requires neither distance-regularity, walk-regularity, nor an association scheme. Regularity \emph{is} used, through the constant eigenvector, and the theorem as stated is not an unweighted bound for arbitrary irregular graphs. Also, the domination condition is entrywise: outside a distance scheme it must be checked from the graph or justified by a walk-counting construction, and does not follow from the adjacency eigenvalues alone. The non-backtracking polynomials in \eqref{eq:non_backtracking} provide such constructions without scheme assumptions, although a particular choice can give a weaker estimate.

Let us next see an example for which one can now directly use the Polynomial EML bound (we do it for the simplest case of $k=1$).

\begin{example}
Replace one vertex of $K_{3,3}$ by a triangle, joining its three new vertices to the three distinct former neighbours. The resulting graph $H$ is cubic on eight vertices. Only the three new vertices lie in a triangle, so the diagonal of $A(H)^3$ is not constant. Since every matrix in a symmetric association-scheme algebra has constant diagonal and that algebra is closed under multiplication, $A(H)$ belongs to no such algebra on this vertex set. Nevertheless, a direct calculation gives
\[
 \det(xI-A(H))=x(x-3)(x^2+x-4)(x^2+x-1)^2.
\]
Taking $p(x)=x$, we have $\beta=(\sqrt{17}-1)/2$, and Theorem~\ref{thm:EML_approach} yields
\[
 \omega(H)\le\frac{8(1+\beta)}{8-3+\beta}
 =\sqrt{17}-1<4.
\]
The inserted triangle therefore gives $\omega(H)=3$. This example uses the Polynomial EML bound directly (instead of the distance-regular equivalence). For a sharp higher-power example beyond distance-regularity, the graphs $H_r=L(S(K_r))$ considered earlier have $r-3$ common neighbours on a within-row edge and none on an edge $(i,j)(j,i)$, so are not distance-regular when $r\ge4$; their diameter-two anticode bound is nevertheless sharp.
\end{example}

\subsection{Applications to coding theory}
The next  examples illustrate how one can now use Proposition~\ref{prop:coding-delsarte} and the spectral machinery we propose to inject alternative spectral arguments and methods to coding theory. The nested-code example uses the more general Theorem~\ref{thm:EML_approach}, while the BCH example shows that a low-degree polynomial can already be sharp.

\medskip
\textbf{Rank-metric anticodes and packing.}
Let $q$ be a prime power and let $2\le a\le b$. The bilinear forms graph $G=\mathrm{Bil}_q(a,b)$ has vertex set $\mathbb F_q^{a\times b}$, with adjacency when the difference has rank one. Its graph distance is $\operatorname{rank}(X-Y)$ and its diameter is $a$. Delsarte's LP method gives the sharp anticode bound~\cite{Delsarte_1978,Tanaka_2006,Tanaka_2012}
\begin{equation}\label{eq:coding-rank}
 \omega_k(\mathrm{Bil}_q(a,b))=L_k(\mathrm{Bil}_q(a,b))=q^{bk}
 \qquad(1\le k<a).
\end{equation}
Equality is realized by the linear space $\mathcal U$ of matrices whose last $a-k$ rows vanish. By Proposition~\ref{prop:coding-delsarte}, the same bound is attained by an adjacency polynomial of degree at most $a$, and it bounds \emph{all} rank-metric anticodes, not only linear ones. For example, $\mathrm{Bil}_2(2,3)$ has $64$ vertices and distinct eigenvalues $21,5,-3$. For $k=1$, the polynomial $p(x)=x$ already gives
\[
 \omega_1(\mathrm{Bil}_2(2,3))\le
 \frac{64(1+5)}{64-21+5}=8,
\]
with equality for $\mathcal U$.

There is also a direct packing consequence. If a matrix code $\mathcal C$ has minimum rank distance at least $k+1$, it meets each additive coset of $\mathcal U$ in at most one matrix. Hence
\[
 |\mathcal C|\le q^{ab}/|\mathcal U|=q^{b(a-k)},
\]
which is the rank-metric Singleton bound. This deduction uses the \emph{constructed} anticode and its cosets, not an upper estimate on $|\mathcal U|$ substituted into a code--anticode inequality. Optimal linear rank-metric anticodes also enter the characterization of generalized rank weights~\cite{Ravagnani_2016}. Thus the equivalence reaches both extremal matrix families and coding-theoretic parameters.

\medskip
\textbf{Nested linear codes.}
Let $C\subseteq D\subseteq\mathbb F_q^\ell$ be linear codes with
$d(C)\ge3$. The coset graph $\Gamma_C$ has the cosets of $C$ as vertices,
with two cosets adjacent when their difference contains a weight-one
word; their distance is the minimum weight in their difference coset.
Its order is $N=q^{\ell-\dim C}$ and its degree is $d=\ell(q-1)$.
A character diagonalization gives its adjacency eigenvalues
\[
 d-q\,\mathrm{wt}(u),\qquad u\in C^\perp:
\]
each zero coordinate contributes $q-1$ to the character sum, and each
nonzero coordinate contributes $-1$. Define
\[
 R_C(D)=\max_{x\in D}\min_{c\in C}\mathrm{wt}(x-c).
\]
If $R_C(D)\le k$, then $D/C$ is a clique of $\Gamma_C^k$, because
$x-y\in D$ for $x,y\in D$. Consequently,
Theorem~\ref{thm:EML_approach} gives the coding bound
\begin{equation}\label{eq:coding-nested}
 q^{\dim D-\dim C}\le
 \frac{N(1+\beta)}{N-p(d)+\beta},\qquad
 \beta=\max_{0\ne u\in C^\perp}p\bigl(d-q\,\mathrm{wt}(u)\bigr),
\end{equation}
whenever $A(\Gamma_C^k)\le p(A(\Gamma_C))$ and the denominator is positive.
Thus dual weights provide the spectral input for a bound on the size of
a linear extension with prescribed relative covering radius.

In particular, $\dim D-\dim C$ is at most the base-$q$ logarithm of the right-hand side of \eqref{eq:coding-nested}, rounded down. The hypothesis $d(C)\ge3$ ensures that the weight-one generators give a simple regular coset graph with the stated degree; no assumption that $C$ is completely regular or that $\Gamma_C$ is distance-regular is needed. Thus \eqref{eq:coding-nested} is an application of the general EML theorem even when the distance relations of $\Gamma_C$ do not form an association scheme. Only when $\Gamma_C$ is distance-regular does Proposition~\ref{prop:coding-delsarte} identify its fully optimized polynomial bound with the anticode LP in the distance scheme.

\medskip
\textbf{A sharp BCH example.}
Let $m\ge3$ be odd, put $s=2^m$, and let $C$ be the binary primitive
narrow-sense double-error-correcting BCH code of length $\ell=s-1$.
Its coset graph $\Gamma$ is distance-regular on $s^2$ vertices, of degree
$\ell$, with nonprincipal eigenvalues
$-1$ and $-1\pm\sqrt{2s}$; see~\cite{vanDam_FonderFlaass_2000}.
Since $d(C)\ge5$, there are no triangles and each pair of vertices at
distance two has exactly two common neighbours. Hence
\[
 A(\Gamma^2)=p(A(\Gamma)),\qquad
 p(x)=x+\frac{x^2-\ell}{2},\qquad
 p(\ell)=\frac{s(s-1)}2,\qquad \beta=\frac s2.
\]
Substitution in Theorem~\ref{thm:EML_approach} yields
\begin{equation}\label{eq:coding-bch}
 \omega_2(\Gamma)\le
 \frac{s^2(1+s/2)}{s^2-s(s-1)/2+s/2}=s=2^m.
\end{equation}
A radius-one ball has $\ell+1=s$ vertices and diameter at most two, so
the bound is attained. In particular, this quadratic spectral bound attains the Delsarte optimum for unrestricted diameter-two
anticodes, and \eqref{eq:coding-nested} gives
$\dim D-\dim C\le m$ whenever $R_C(D)\le2$.

Note that for this example $A(\Gamma^2)=p(A(\Gamma))$ gives the required entrywise domination, the nonprincipal eigenvalues determine $\beta$, and the radius-one ball shows that the bound is sharp. Thus the full adjacency algebra is needed for the general equivalence, but in this example a quadratic polynomial is enough. Indeed, outside distance schemes, our Polynomial EML bound still applies.

\section*{Declaration of AI use}
Most of the core ideas in this paper were developed by the authors during the Algebraic Graph Theory and Applications workshop, held in  
October 2025 at the 
Banff International Research Station, Canada. The authors acknowledge the use of ChatGPT (GPT-5.6, OpenAI; accessed September 2026) solely for further exploration and improvements. AI tools were not used to draft the manuscript. All formal statements, arguments, and proofs in the manuscript were written and checked by the authors, who take full responsibility for the accuracy and integrity of the article. 

\section*{Acknowledgements}
The authors thank the organizers of the Algebraic Graph Theory and Applications workshop held at the BIRS location in Banff, and also, in particular, the participants Seyed Alireza,  Himanshu Gupta, Frederico Cançado, Shahla Nasserasr, Shivaramakrishna Pragada, and Ralihe Villagr\'an for discussions on this topic.

Aida Abiad is supported by NWO (Dutch Research Council) through the grant VI.Vidi.213.085. This work was done in part while the Aida Abiad was visiting the Simons Institute for the Theory of Computing at UC Berkeley. Gabriel Coutinho and Emanuel Juliano are supported by FAPEMIG and CNPq grants. Michael Tait is partially supported by NSF grant DMS-2245556, a Villanova University Summer Grant, and a Villanova CLAS Research Semester. 
\bibliographystyle{plain}
\bibliography{references}

\affl{Aida Abiad}{a.abiad.monge@tue.nl}{Department of Mathematics and Computer Science, Eindhoven University of Technology, The Netherlands \\ Department of Mathematics and Data Science, Vrije Universiteit Brussel, Belgium}

\affl{Gabriel Coutinho}{gabriel@dcc.ufmg.br}{Department of Computer Science, Federal University of Minas Gerais, Brazil}

\affl{Hitesh Kumar}{hitesh.kumar.math@gmail.com, hitesh\_kumar@sfu.ca}{Department of Mathematics, Simon Fraser University, Canada}

\affl{Emanuel Juliano}{emanuelsilva@dcc.ufmg.br}{Department of Computer Science, Federal University of Minas Gerais, Brazil}

\affl{Michael Tait}{michael.tait@villanova.edu}{Department of Mathematics \& Statistics, Villanova University, U.S.A.}

\end{document}